\documentclass[11pt]{article}

\usepackage[margin=1in]{geometry}
\usepackage{amsthm}
\usepackage{amssymb}
\usepackage{amsmath}
\usepackage{graphicx}
\usepackage{url}
\usepackage{color}
\usepackage{framed}
\usepackage[table]{xcolor}

\newcommand\blfootnote[1]{%
  \begingroup
  \renewcommand\thefootnote{}\footnote{#1}%
  \addtocounter{footnote}{-1}%
  \endgroup
}

\newtheorem{theorem}{Theorem}

\newtheorem{lemma}[theorem]{Lemma}

\theoremstyle{definition}

\definecolor{sand}{HTML}{CCD3DB}
\definecolor{or}{HTML}{F58634}
\definecolor{sea}{HTML}{2E7FAB}
\definecolor{lag}{HTML}{F2DBDC}
\definecolor{lg}{gray}{0.85}
\definecolor{dg}{gray}{0.33}

\title{Utility Ghost:\ Gamified redistricting with partisan symmetry}

\author{Dustin~G.~Mixon\footnote{Department of Mathematics, The Ohio State University, Columbus, OH}\qquad Soledad Villar\footnote{Center for Data Science, New York University, New York, NY}}
\date{}

\begin{document}
\maketitle

\blfootnote{Send correspondence to \texttt{mixon.23@osu.edu}}

\begin{abstract}
Inspired by the word game Ghost, we propose a new protocol for bipartisan redistricting in which partisan players take turns assigning precincts to districts.
We prove that in an idealized setting, if both parties have the same number votes, then under optimal play in our protocol, both parties win the same number of seats.
We also evaluate our protocol in more realistic settings that show how our game nearly eliminates the first-player advantage exhibited in other redistricting protocols.
\end{abstract}

\section{Introduction}

Legislators across the United States are elected from voting districts, and these voting districts are modified every ten years to reflect changes in population distribution, as measured by the U.S.\ Census.
Redistricting is the responsibility of state governments, and if a given state government is led by either major political party, then the resulting redistricting may be drawn to benefit that party.
This effect has been called \textbf{partisan gerrymandering} ever since Governor Elbridge Gerry of Massachusetts approved a redrawing of voting districts in 1812---one district resembled the profile of a salamander, as famously lambasted by a political cartoon in the \textit{Boston Gazette}~\cite{Griffith:07}.
Over 200 years later, partisan gerrymandering continues to be a significant issue.
In 2017, the plaintiffs in U.S.\ Supreme Court case \textit{Gill v.\ Whitford} used the so-called efficiency gap statistic~\cite{StephanopoulosM:15} to argue that the voting districts for the Wisconsin State Assembly were drawn in a way that disproportionally wastes votes that were cast by Democrats.
In 2018, the Pennsylvania Supreme Court struck down the 2011 congressional map as an unconstitutional partisan gerrymander; here, the plantiffs leveraged Markov Chain--based techniques to sample the space of admissible maps and make informative comparisons with the map in question~\cite{ChikinaFP:17}. 
This trend of policing partisan gerrymandering in the courts has led to a flurry of relevant mathematical research, primarily in Markov Chain Monte Carlo sampling methods~\cite{FifieldHIT:15,BangiaEtal:17,HerschlangRM:17,HerchlangEtal:18,TamChoL:18} and in evaluating various fairness criteria~\cite{BernsteinD:17,Soberon:17,AlexeevM:17a,AlexeevM:17b,KuengMV:18}.
However, in light of recent changes to the U.S.\ Supreme Court bench, a new approach might be necessary to effectively combat partisan gerrymandering in the future.

Along these lines, one prominent approach is to \textit{prevent} partisan gerrymandering by changing the way maps are drawn in the first place.
For example, several states now delegate mapmaking powers to a non-partisan or bipartisan redistricting commission so as to minimize the effects of partisan politics.
In order for such a redistricting commission to succeed, one must either identify sufficiently independent agents, or design a protocol that ensures an independent outcome despite partisan inputs.
Pegden, Procaccia and Yu~\cite{PegdenPY:17} recently proposed one such protocol based on ``I-cut-you-choose'' solutions to classical cake-cutting problems.
Given the task of partitioning a state into $k$ districts, their \textit{I-cut-you-freeze} protocol is a game in which two partisan players make moves that iteratively determine a partition.
The first player draws an admissible map of $k$ proposed districts and passes the map to the second player.
For each subsequent turn, a player ``freezes'' one of the proposed districts in the map she's given, and then redraws the remainder of the map into proposed districts before passing the map back to the other player.
Each of these moves freezes an additional district, and so gameplay ends after $k$ such moves.
The main result in~\cite{PegdenPY:17} characterizes the number of districts won under optimal play in the case where districts have no geometric constraints (e.g., each district need not be contiguous).
Under optimal play, this protocol produces districts that exhibit partisan symmetry as $k\to\infty$, but for the smaller $k$ that we encounter in the real world, the I-cut-you-freeze protocol gives substantial advantage to the first player.
For example, there are currently five states with only $k=2$ U.S.\ congressional districts, in which case the first player serves as \textit{de facto} mapmaker, while the second player has no input in the process.

In this paper, we propose an alternative game that does not suffer from such first-player advantage.
Taking inspiration from the word game Ghost, we design a game in which partisan players take turns assigning precincts to districts.
In the non-geometric setting, we prove that both players win the same number of districts when they have the same number of votes.
We also test our game in the geometric setting, where we observe that our game produces districts that exhibit partisan symmetry.
When optimally playing our game to redistrict New Hampshire (based on election returns from the 2016 presidential election), we achieve proportional results---no matter which partisan player moves first.
In the following section, we introduce our game and discuss how it performs in the geometric setting.
Section~3 contains our theoretical result, and we discuss our results in Section~4.

\section{Utility Ghost and bipartisan redistricting}

Our protocol is motivated by the word game \textit{Ghost}:
Two players alternate presenting letters, and the first to either spell an English word or create a string that does not start a word loses.
We introduce a variant called \textbf{Utility Ghost}, which requires a language $L$ and utility functions $u_1,u_2\colon L\to\mathbb{R}$.
(We only consider finite languages.)
As in Ghost, players alternate presenting letters.
If either player creates a string that does not start a word in $L$, he receives $-\infty$ utility.
Otherwise, one of the players eventually spells a word $w\in L$, in which case player $1$ (i.e., the player who made the first move) receives utility $u_1(w)$ and player $2$ receives utility $u_2(w)$.
The goal of each player is to maximize utility.
As an example, letting $E$ and $S$ denote the sets of English and Spanish words consisting of at least three letters, one could play Utility Ghost with language $L=E\cup S$ and utility functions defined by
\[
u_1(w)
=-u_2(w)
=\left\{\begin{array}{rl}
1&\text{if }w\in E\setminus S\\
0&\text{if }w\in E\cap S\\
-1&\text{if }w\in S\setminus E.
\end{array}\right.
\]
In this instance of Utility Ghost,\footnote{Almost every first move from the English alphabet has a response in $\{\text{\'{e}},\text{\~{n}},\text{\'{o}}\}$, with the exception of q. Under optimal play, the first player wins this game since the word ``qubit'' can be forced.} the first player attempts to spell an English word, while the second player attempts to spell a Spanish word.
Overall, the rules of gameplay in Utility Ghost are effectively the same as in Ghost---all that has changed are the players' objectives.

We propose Utility Ghost as a gamification of bipartisan redistricting.
Suppose a bipartisan commission is tasked with partitioning $n$ labeled \textit{atoms} into $k$ admissible districts; here, atoms may be census blocks, precincts, or counties, for example.
We consider a language $L\subseteq\Sigma^n$ over the alphabet $\Sigma=[n]\times[k]$.
Here, the symbol $(a,b)\in\Sigma$ should be interpreted as instructions of the form ``assign atom $a$ to district $b$.''
In practice, districts are required to satisfy several constraints, and so we define $L$ to correspond to the choices of districts that satisfy these constraints.
Split the bipartisan commission into two players according to party, and then have them play Utility Ghost to partition a state into districts.
In words, partisan players take turns assigning atoms to districts.\footnote{In a real-world implementation of this game, the player also provides an admissible completion of the current map in order to demonstrate that the selected assignment is allowed.}
In this setup, the $i$th utility function counts the number of districts carried by the $i$th player's party for $i\in\{1,2\}$.
This definition of utility requires perfect information about voter preference, which can be estimated in practice using prior election data.

For the sake of illustration, we play out a small redistricting instance of Utility Ghost.
For this game, the task is to partition a state into two contiguous districts, each composed of three counties.
Take the atoms to be counties, which have the same number of voters.
Four counties vote unanimous-red, and the other two vote unanimous-blue.
Suppose the first player ($P_1$) seeks to maximize the number of districts that are majority-red, while the second player ($P_2$) maximizes majority-blue districts.
What follows is an example of optimal play in this setting:
\[
\begin{array}{|c|c|}
\hline
\cellcolor{red} & \phantom{\mathbf{2}} \cellcolor{red}  \\ \hline
\cellcolor{red}  &  \cellcolor{red}  \\ \hline
\cellcolor{blue} \phantom{\mathbf{2}} &  \cellcolor{blue}  \\ \hline
\end{array}
\stackrel{P_1}{\longrightarrow}
\begin{array}{|c|c|}
\hline
\cellcolor{red} \textcolor{white}{\mathbf{1}} & \phantom{\mathbf{2}} \cellcolor{red}  \\ \hline
\cellcolor{red}  &  \cellcolor{red}  \\ \hline
\cellcolor{blue} \phantom{\mathbf{2}} &  \cellcolor{blue}  \\ \hline
\end{array}
\stackrel{P_2}{\longrightarrow}
\begin{array}{|c|c|}
\hline
\cellcolor{red} \textcolor{white}{\mathbf{1}} & \phantom{\mathbf{2}} \cellcolor{red}  \\ \hline
\cellcolor{red}  &  \cellcolor{red}  \\ \hline
\cellcolor{blue} \textcolor{white}{\mathbf{2}} &  \cellcolor{blue}  \\ \hline
\end{array}
\stackrel{P_1}{\longrightarrow}
\begin{array}{|c|c|}
\hline
\cellcolor{red} \textcolor{white}{\mathbf{1}} & \phantom{\mathbf{2}} \cellcolor{red}  \\ \hline
\cellcolor{red}  & \textcolor{white}{\mathbf{2}} \cellcolor{red}  \\ \hline
\cellcolor{blue} \textcolor{white}{\mathbf{2}} &  \cellcolor{blue}  \\ \hline
\end{array}
\stackrel{P_2}{\longrightarrow}
\begin{array}{|c|c|}
\hline
\cellcolor{red} \textcolor{white}{\mathbf{1}} &  \cellcolor{red}  \\ \hline
\cellcolor{red}  &  \cellcolor{red} \textcolor{white}{\mathbf{2}} \\ \hline
\cellcolor{blue} \textcolor{white}{\mathbf{2}} &  \cellcolor{blue} \textcolor{white}{\mathbf{2}} \\ \hline
\end{array}
\stackrel{P_1}{\longrightarrow}
\begin{array}{|c|c|}
\hline
\cellcolor{red} \textcolor{white}{\mathbf{1}} &  \cellcolor{red} \textcolor{white}{\mathbf{1}} \\ \hline
\cellcolor{red}  &  \cellcolor{red} \textcolor{white}{\mathbf{2}} \\ \hline
\cellcolor{blue} \textcolor{white}{\mathbf{2}} &  \cellcolor{blue} \textcolor{white}{\mathbf{2}} \\ \hline
\end{array}
\stackrel{P_2}{\longrightarrow}
\begin{array}{|c|c|}
\hline
\cellcolor{red} \textcolor{white}{\mathbf{1}} &  \cellcolor{red} \textcolor{white}{\mathbf{1}} \\ \hline
\cellcolor{red} \textcolor{white}{\mathbf{1}}  &  \cellcolor{red} \textcolor{white}{\mathbf{2}} \\ \hline
\cellcolor{blue} \textcolor{white}{\mathbf{2}} &  \cellcolor{blue} \textcolor{white}{\mathbf{2}} \\ \hline
\end{array}
\]
Notice that since both districts must be contiguous, $P_2$'s first move ensures that both blue counties will end up residing in district 2, thereby securing a majority-blue district.

In the next section, we solve redistricting instances of Utility Ghost in which atoms are voters (all of whom have known preference), and districts are only required to contain the same number of voters.
In particular, we prove that both players win the same number of districts when they have the same number of votes.
But how are the results of this game impacted by the sort of geometric constraints we encounter in the real world?
When the game is sufficiently small, it can be naively solved using the standard minimax algorithm.
We followed this approach in two settings:

\subsection{Redistricting a decomino state}

Consider the decomino state depicted in Figure~\ref{fig.nm}, which is made up of 10 square counties of equal size.
Assume that each county contains the same sized population and number of voters, and each county unanimously votes for either $A$ or $B$.
Our task is to partition the decomino state into two contiguous districts of five counties.
Exactly seven maps satisfy these constraints.
When we randomly draw a voter distribution, we may solve the corresponding instance of Utility Ghost (with $A$ as first player) and count the resulting number of majority-$A$ districts.
On average, we obtain a symmetric votes--seats curve (depicted in Figure~\ref{fig.nm}), whereas the curve arising from gerrymandering for $A$ (or $B$) is particularly asymmetric.
Note that since there are only two districts, these gerrymandered results would arise from the I-cut-you-freeze protocol proposed in~\cite{PegdenPY:17} with first player $A$ ($B$, respectively).

\begin{figure}
\begin{center}
\includegraphics[width=0.13\textwidth,trim={3cm 1cm 3cm 1cm},clip]{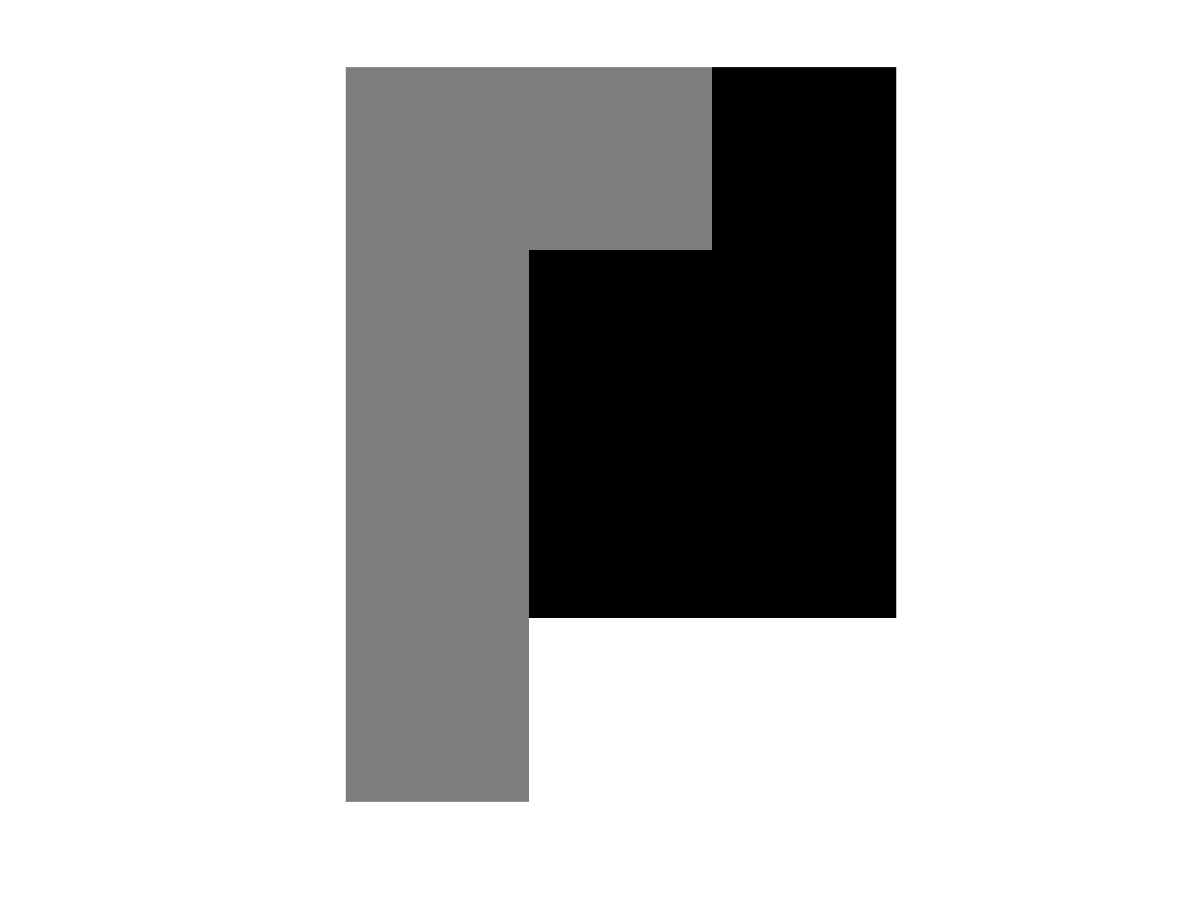}
\includegraphics[width=0.13\textwidth,trim={3cm 1cm 3cm 1cm},clip]{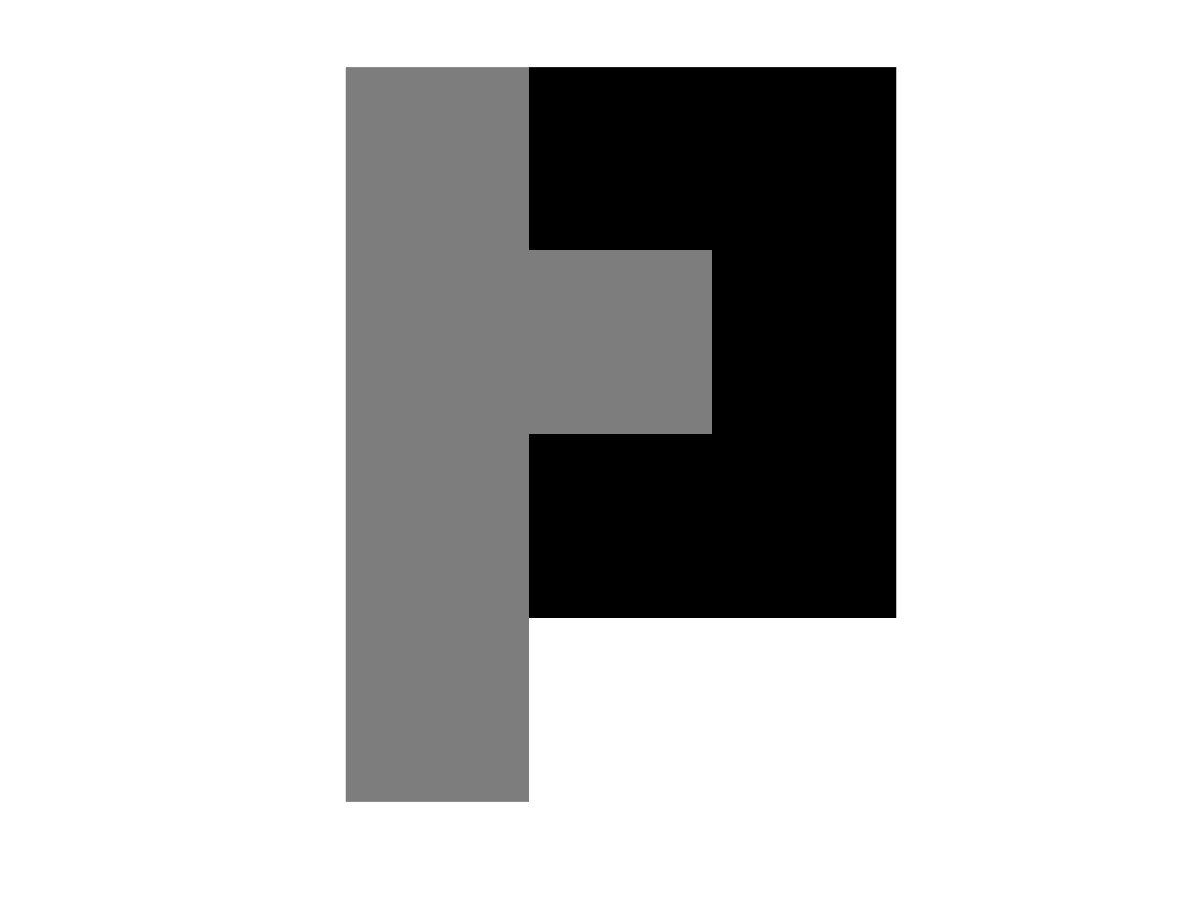}
\includegraphics[width=0.13\textwidth,trim={3cm 1cm 3cm 1cm},clip]{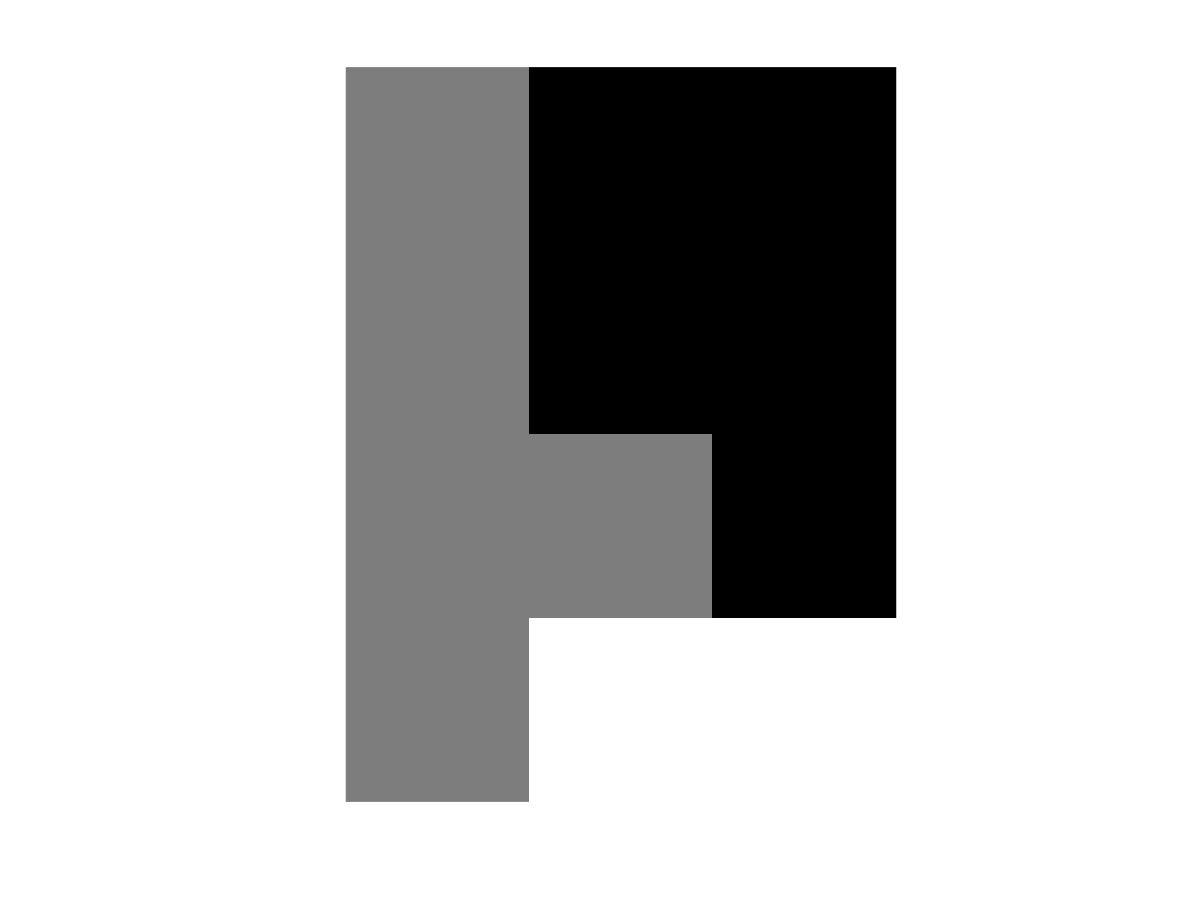}
\includegraphics[width=0.13\textwidth,trim={3cm 1cm 3cm 1cm},clip]{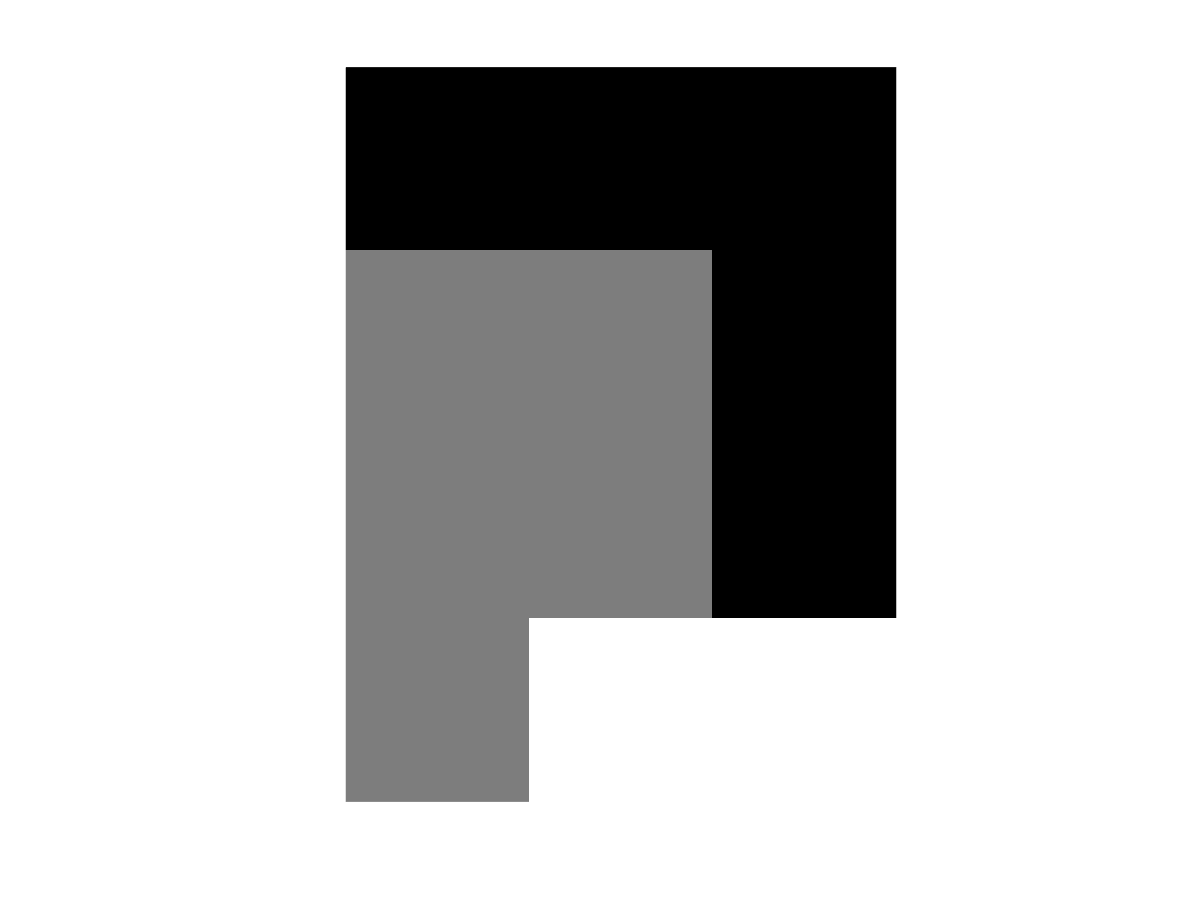}
\includegraphics[width=0.13\textwidth,trim={3cm 1cm 3cm 1cm},clip]{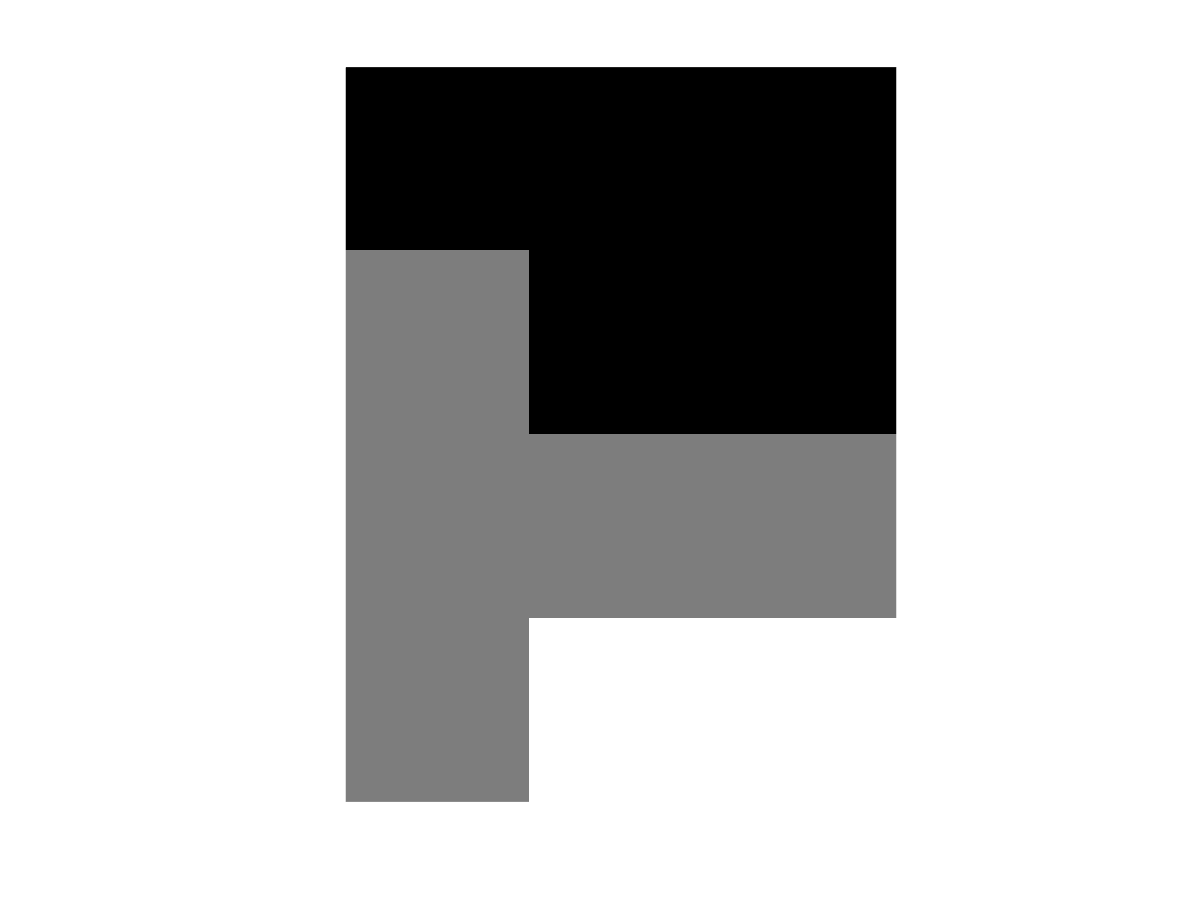}
\includegraphics[width=0.13\textwidth,trim={3cm 1cm 3cm 1cm},clip]{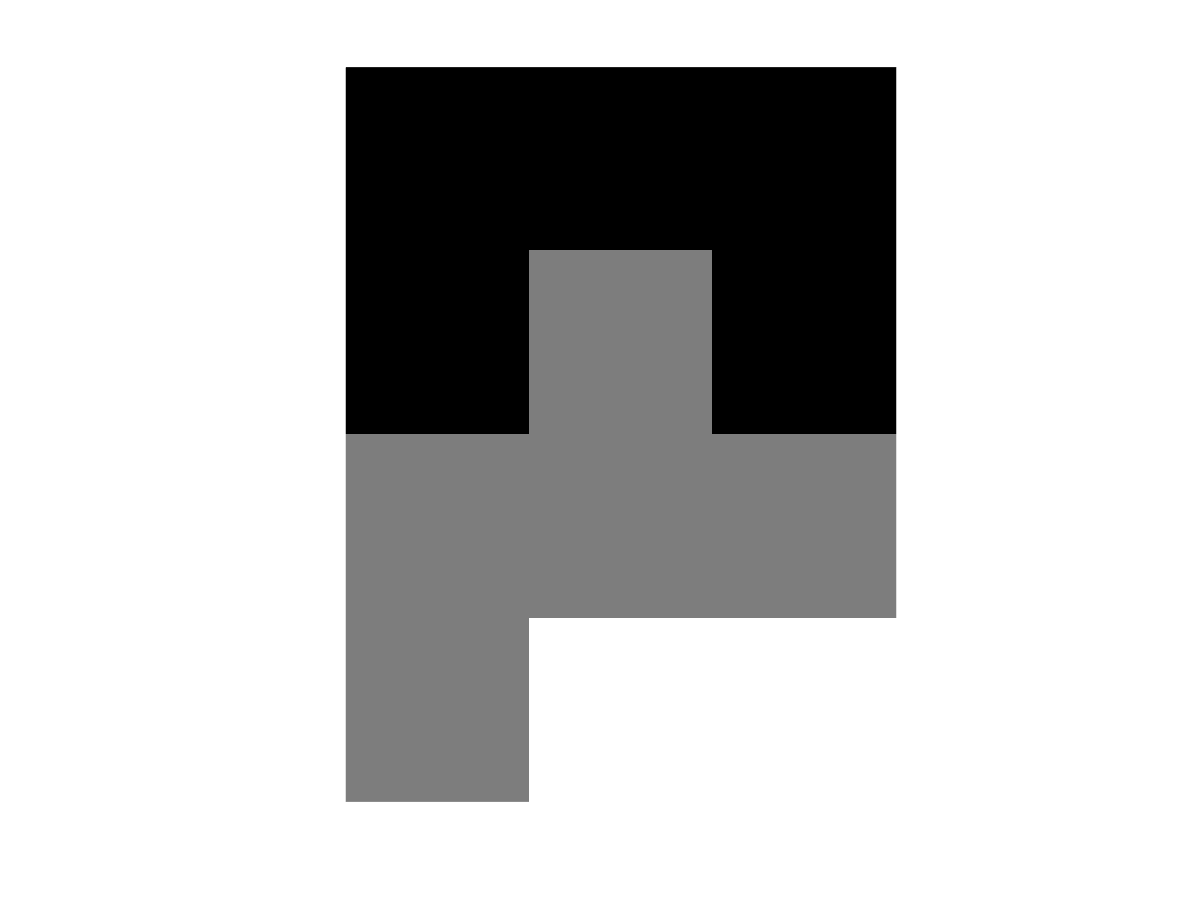}
\includegraphics[width=0.13\textwidth,trim={3cm 1cm 3cm 1cm},clip]{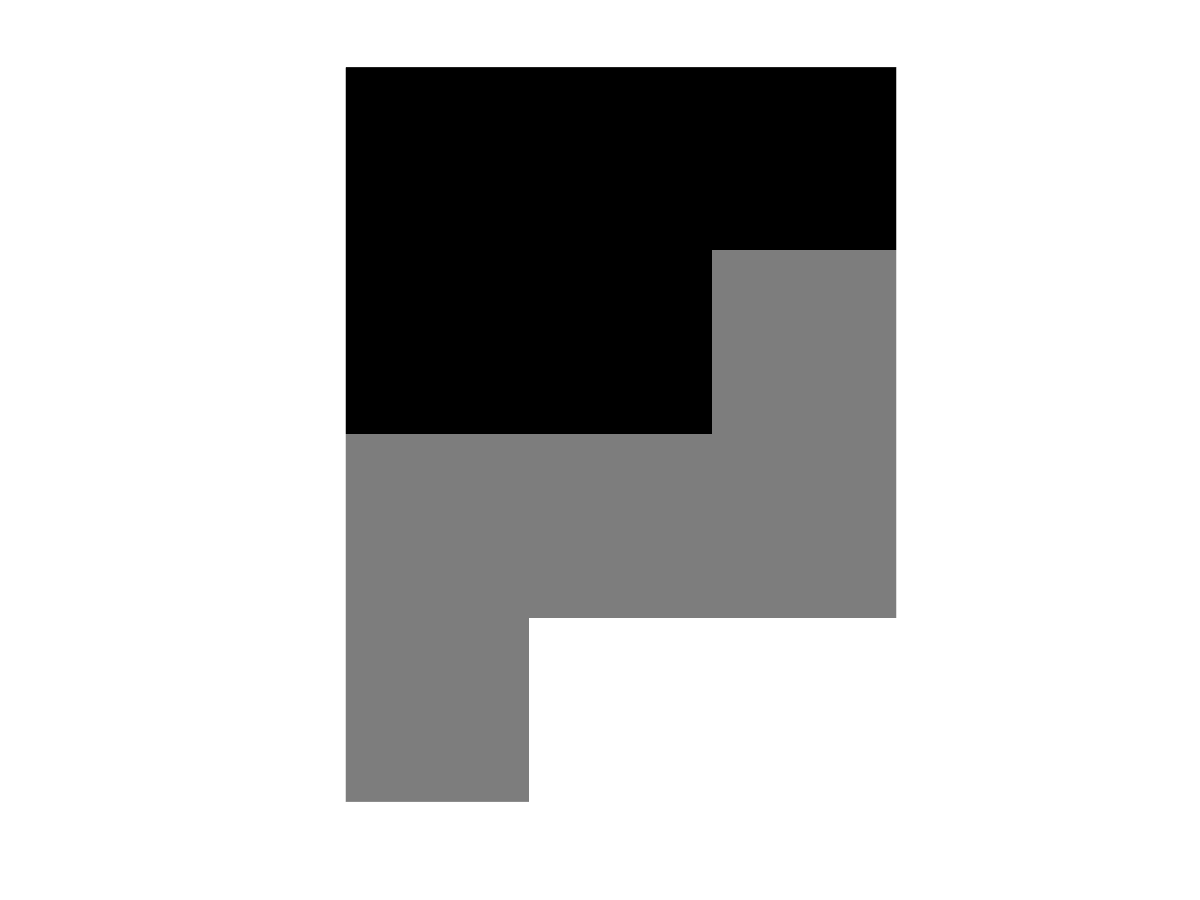}\\

\bigskip

\bigskip

\hspace{1.1cm}\footnotesize{random maps}\hspace{1.6cm}
\footnotesize{maps favoring $A$}\hspace{1.5cm}
\footnotesize{maps favoring $B$}\hspace{1.25cm}
\footnotesize{Utility Ghost maps}\hspace{0.25cm}\\
\rotatebox{90}{\hspace{0.4cm}\footnotesize{number of majority-$A$ districts}}
\includegraphics[width=0.23\textwidth,trim={5cm 1cm 5cm 0.8cm},clip]{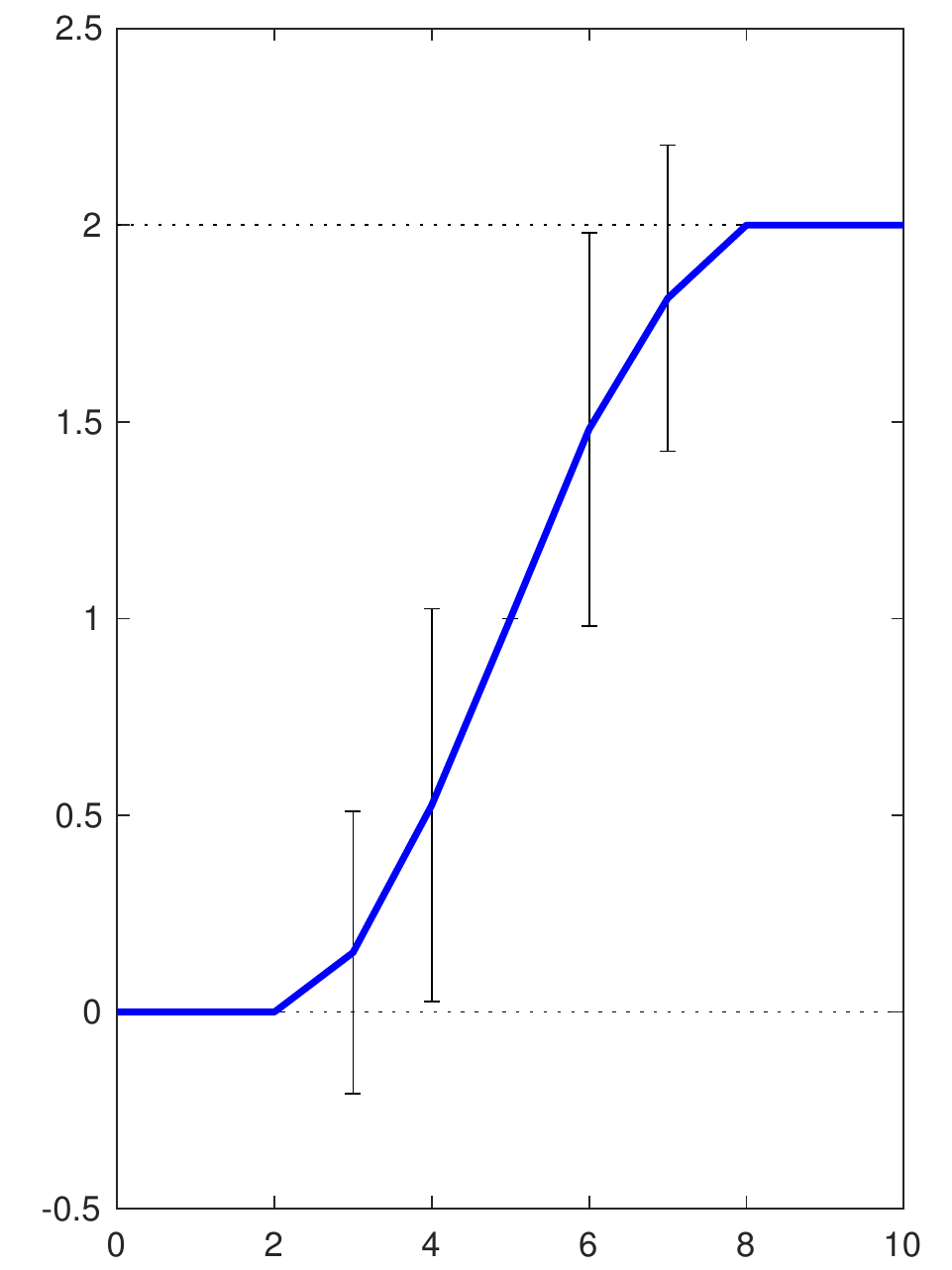}
\includegraphics[width=0.23\textwidth,trim={5cm 1cm 5cm 0.8cm},clip]{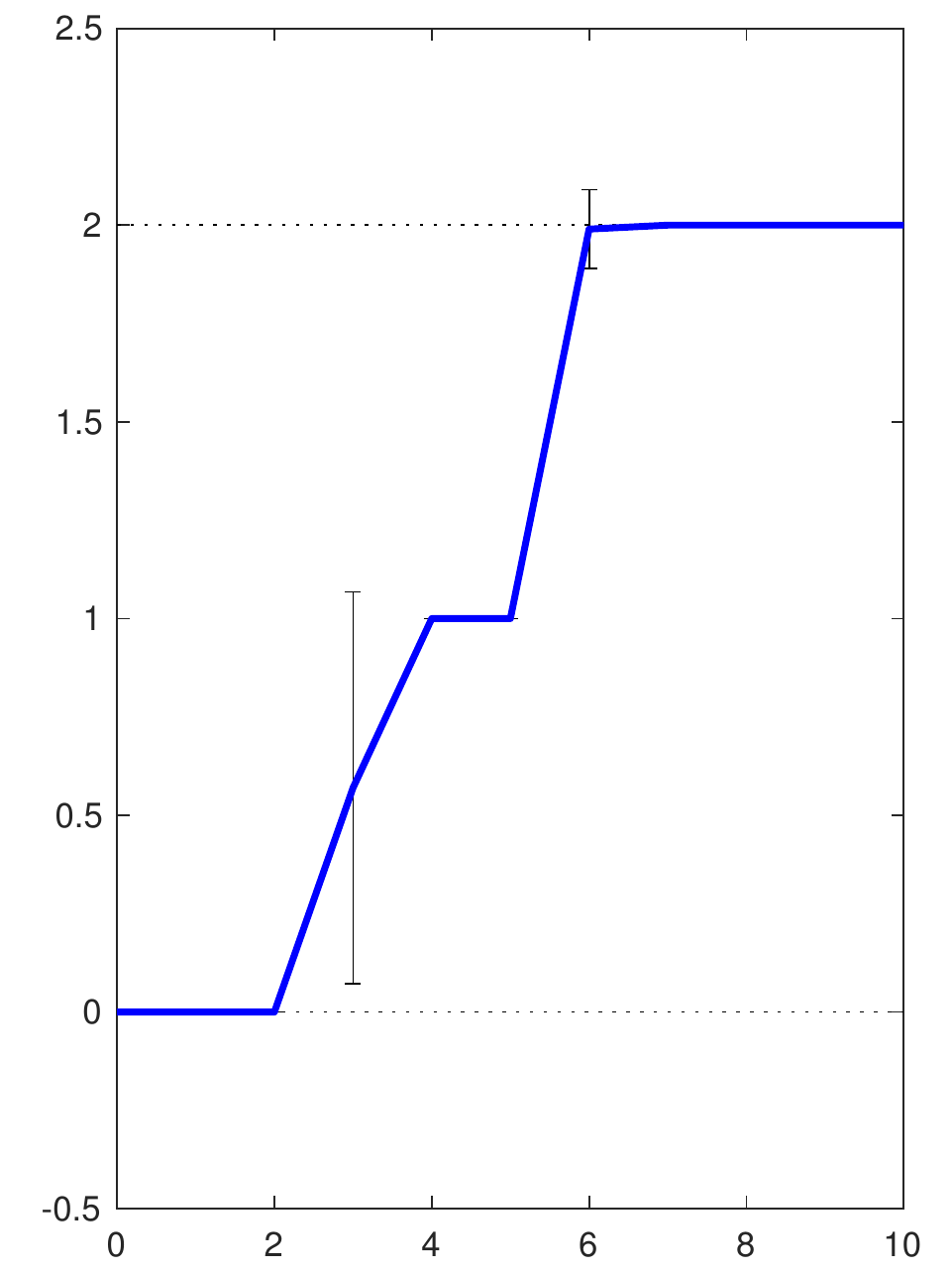}
\includegraphics[width=0.23\textwidth,trim={5cm 1cm 5cm 0.8cm},clip]{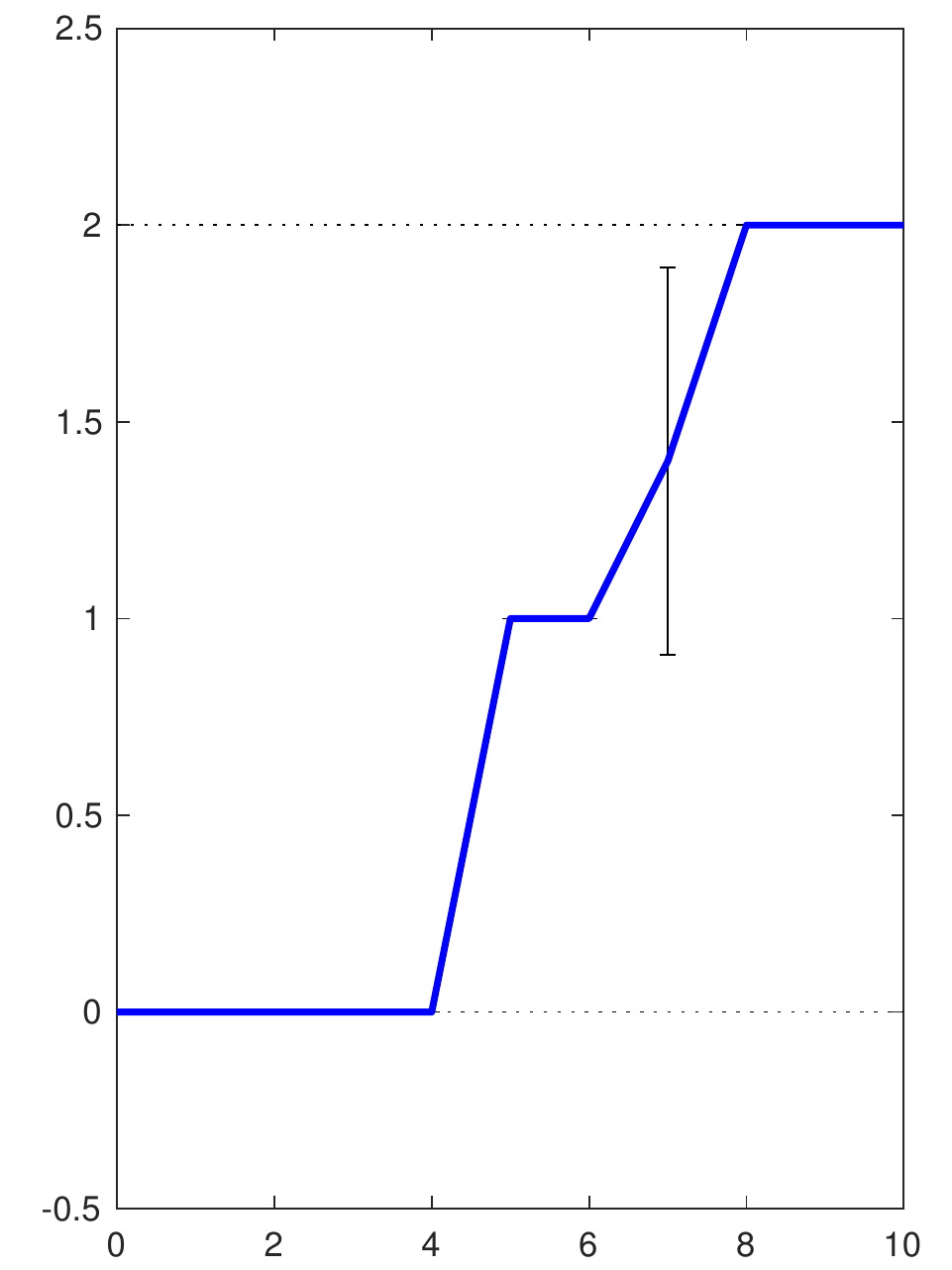}
\includegraphics[width=0.23\textwidth,trim={5cm 1cm 5cm 0.8cm},clip]{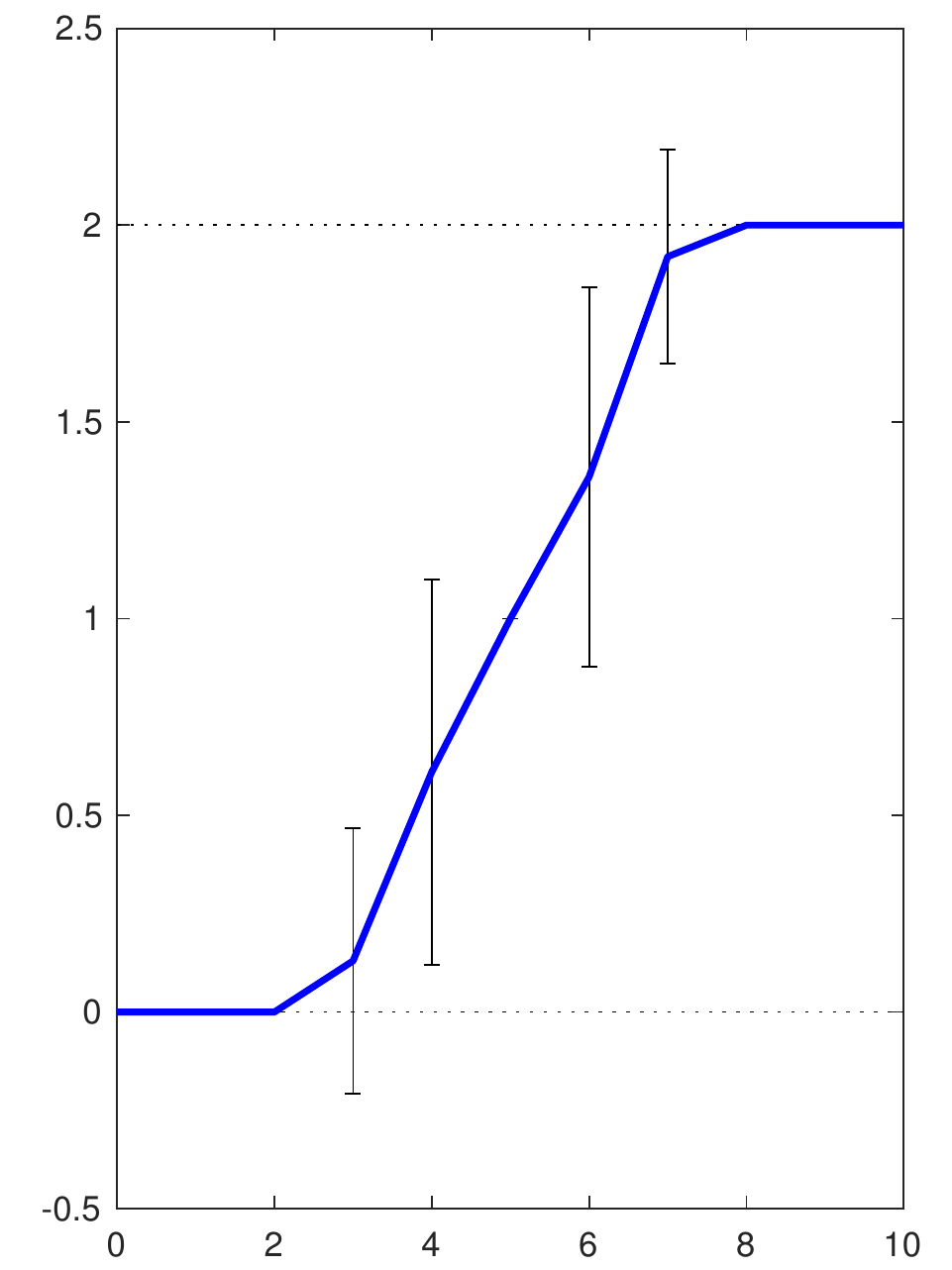}\\
\footnotesize{number of unanimous-$A$ counties}
\end{center}
\caption{\label{fig.nm}
\textbf{(top)}
Consider a state consisting of 10 unit-square counties that must be divided into two voting districts.
By assumption, each county contains the same sized population.
There are only seven districtings that do not split any county while satisfying contiguity and one person--one vote.
The geometry of the state imposes nontrivial structure in these districtings.
For example, the top-right and bottom-left counties are always assigned to different districts.
\textbf{(bottom)}
Let $A$ and $B$ denote the two major parties.
Assume that each county unanimously votes for either $A$ or $B$ and with equal voter turnout.
Fix $x\in\{0,\ldots,10\}$, and consider all possible voter distributions with exactly $x$ unanimous-$A$ counties.
Draw such a voter distribution uniformly at random and perform four experiments:
(1) collect the number of majority-$A$ districts for each of the seven districtings;
(2) identify the maximum number of majority-$A$ districts among the seven districtings;
(3) identify the minimum number of majority-$A$ districts among the seven districtings; and
(4) play Utility Ghost optimally (with $A$ making the first move) to select a districting and then compute the resulting number of majority-$A$ districts.
Perform this experiment for $100$ independent random draws of voter distributions for each $x\in\{0,\ldots,10\}$.
The above plots illustrate the sample mean and standard deviation of (1)--(4), respectively.
In words, the far-left plot represents a votes--seats curve for a random districting, the middle-left plot corresponds to gerrymandering for $A$, middle-right to gerrymandering for $B$, and far-right to the outcome of Utility Ghost.
As expected, random admissible districtings exhibit partisan symmetry.
Utility Ghost appears to mimic this same notion of partisan symmetry, whereas the I-cut-you-freeze protocol~\cite{PegdenPY:17} would produce the gerrymandered results in this case.
}
\end{figure}

\subsection{Redistricting New Hampshire}

New Hampshire is made up of 10 counties and currently has two U.S.\ Congressional voting districts.
If we require each district to be contiguous and comprised of counties, and furthermore, that the difference in district populations be less than 10 percent of the entire population (according to the 2010 U.S.\ Census), then there are seven admissible maps.
Once we know the voter distributions of each county, then we can solve this instance of Utility Ghost.
Here, we used the 2016 presidential election returns as a proxy for voter distribution.
This election was very close, with Hillary Clinton receiving 47.62 percent of the vote and Donald Trump receiving 47.25 percent.
It therefore comes as no surprise that six out of the seven admissible maps produce proportional results under this proxy.
Still, one of the admissible maps gives both seats to the Democrats, and this would be the map selected by the I-cut-you-freeze protocol~\cite{PegdenPY:17} if Democrats play first.
Thankfully, optimal play under the Utility Ghost protocol avoids this map, regardless of which party plays first.

\begin{figure}
\begin{center}
\includegraphics[width=0.11\textwidth,trim={5cm 1cm 5cm 1cm},clip]{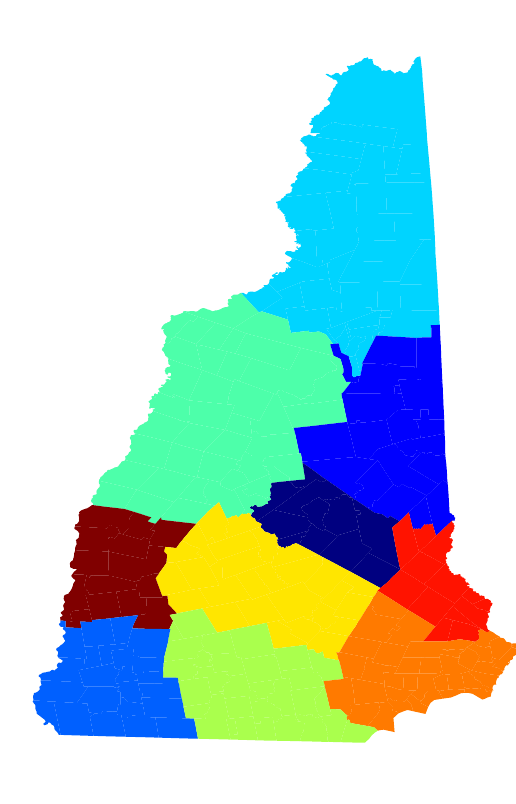}
\includegraphics[width=0.11\textwidth,trim={5cm 1cm 5cm 1cm},clip]{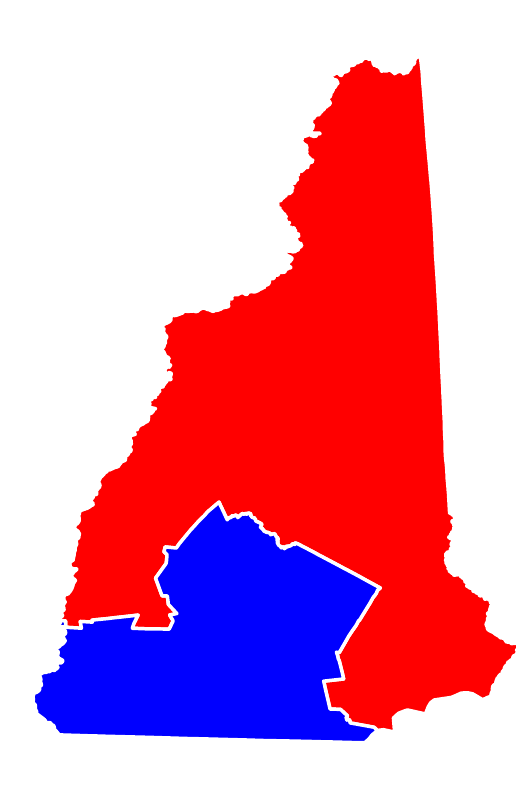}
\includegraphics[width=0.11\textwidth,trim={5cm 1cm 5cm 1cm},clip]{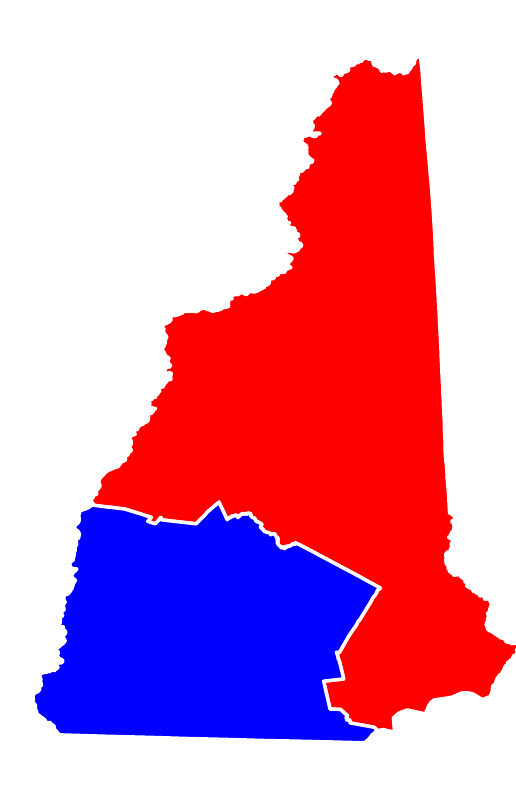}
\includegraphics[width=0.11\textwidth,trim={5cm 1cm 5cm 1cm},clip]{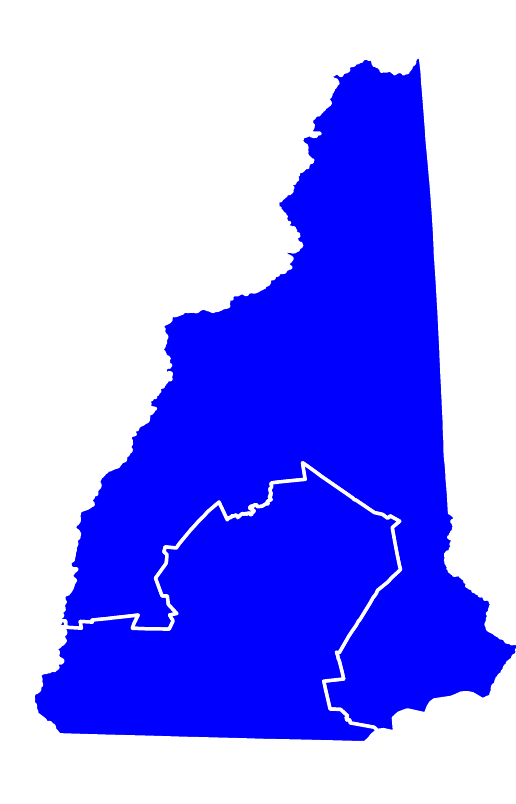}
\includegraphics[width=0.11\textwidth,trim={5cm 1cm 5cm 1cm},clip]{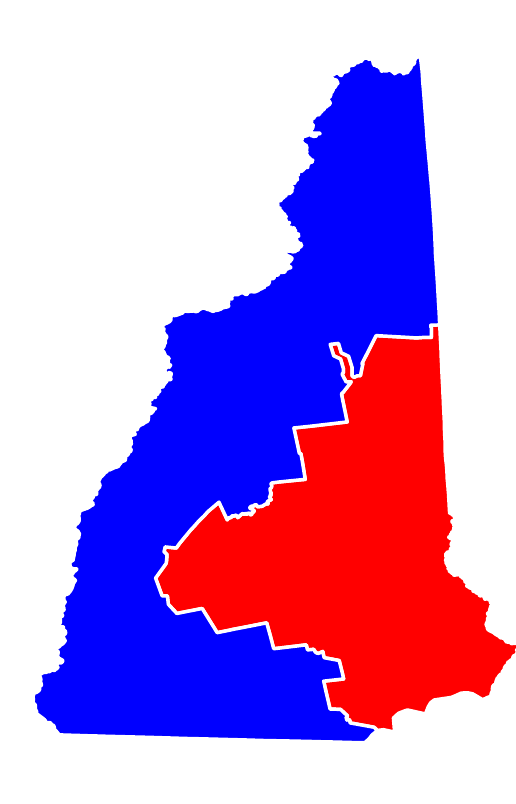}
\includegraphics[width=0.11\textwidth,trim={5cm 1cm 5cm 1cm},clip]{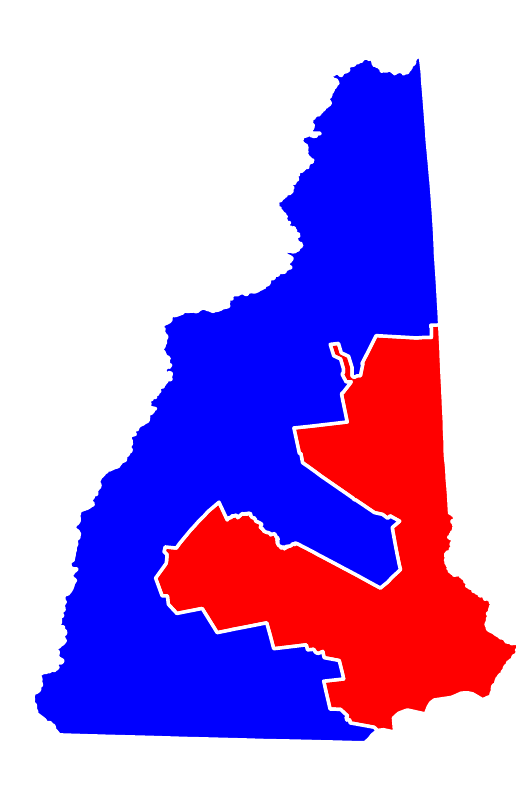}
\includegraphics[width=0.11\textwidth,trim={5cm 1cm 5cm 1cm},clip]{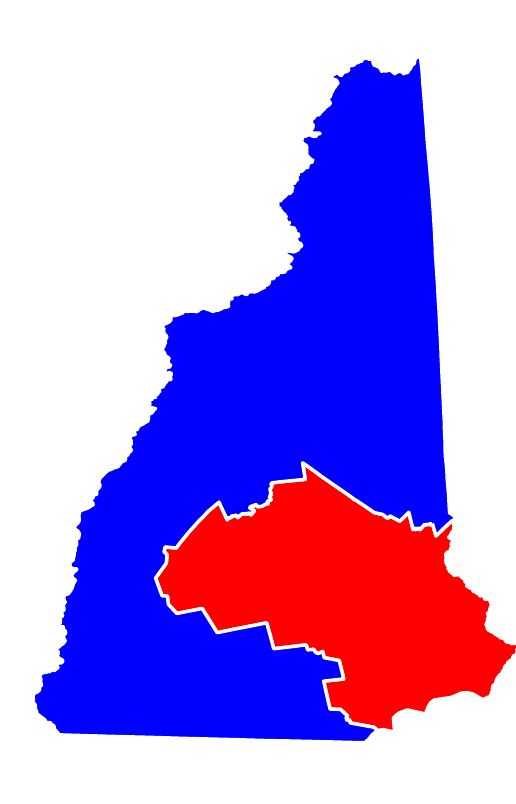}
\includegraphics[width=0.11\textwidth,trim={5cm 1cm 5cm 1cm},clip]{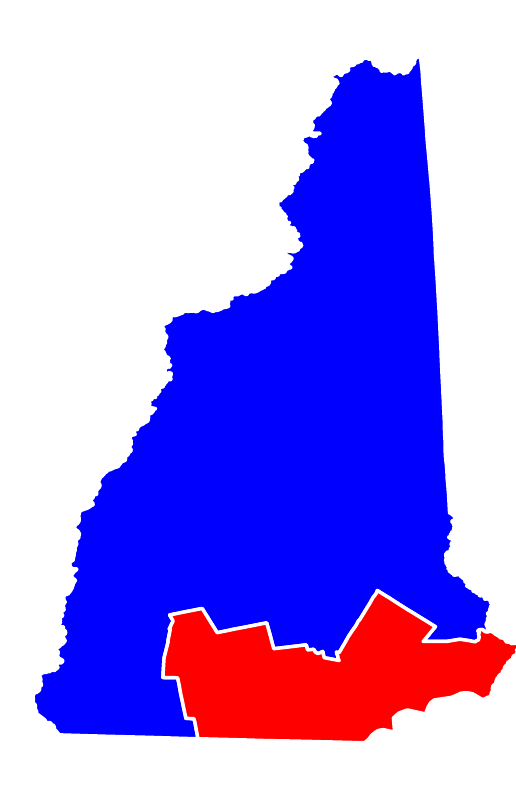}
\end{center}
\caption{\label{fig.nh}
New Hampshire is made up of 10 counties, illustrated on the left.
Currently, New Hampshire has two U.S.\ congressional districts.
Consider all possible partitions of New Hampshire into two contiguous districts such that none of the counties are split, and such that the difference in district populations is less than 10 percent of the entire population (according to the 2010 U.S.\ Census).
There are seven such partitions, illustrated above.
Here, we color the districts according to which party received more votes in the 2016 presidential election.
Using this election data as a proxy for voter preferences, one of the seven admissible districtings results in two congressional seats for the Democrats, whereas the remaining six give a proportional result.
If the Democrats play first in the I-cut-you-freeze protocol~\cite{PegdenPY:17}, then they get to select the districting that wins both seats.
On the other hand, in the Utility Ghost protocol we introduced, Republicans avoid this worst-case map under optimal play, regardless of which party plays first.
}
\end{figure}

\section{Theory}

In this section, we focus on the redistricting instances of Utility Ghost in which atoms are voters (all of whom have known preference), and the districts are only required to contain the same number of voters (i.e., we impose no geometric constraints on the districts).
Furthermore, exactly half of the voters have preference for one party, while the other half have preference for the other party.
How many districts do partisan players win under optimal play?

It is convenient to abstract this particular instance of Utility Ghost in terms of balls and bins.
Fix positive integers $j$ and $m$.
Two players take turns placing one of $n=2j(2m+1)$ balls into one of $k=2j$ bins, each of capacity $2m+1$.
Half of the balls are white, and the other half are black.
By design, $n$ and $k$ are both even, and each bin has odd capacity so that the notion of majority is unambiguous.
The first player ($P_1$) wins a bin if it contains a majority of white balls (at least $m+1$), otherwise the second player ($P_2$) wins that bin.
At the end of the game, the player with the most bins wins.
If both players win half (i.e., $j$) of the bins, the game ends in a tie.
We refer to this game as the \textbf{$(j,m)$-balanced redistricting game}.

\begin{theorem}[Main result]
Suppose $m>14j$.
Then under optimal play, the $(j,m)$-balanced redistricting game ends in a tie.
\end{theorem}

\begin{lemma}
In the $(j,m)$-balanced redistricting game, $P_2$ can force a tie.
\end{lemma}

\begin{proof}
Arbitrarily label the bins with $\{(a,b)\}_{a\in\{\pm1\},b\in[j]}$.
We propose the following strategy for $P_2$:
At each round, $P_1$ puts a ball of some color in bin $(a,b)$, and $P_2$ responds by putting a ball of the opposite color in bin $(-a,b)$.
We claim that such moves are always possible, and that under this play, the game ends in a tie.
To this end, a simple induction argument gives both of the following:
\begin{itemize}
\item[(i)]
After each round, the number of white balls remaining equals the number of black balls remaining.
\item[(ii)]
After each round, the number of white/black balls in $(a,b)$ equals the number of black/white balls in $(-a,b)$ for every $a\in\{\pm1\}$ and $b\in[j]$.
\end{itemize}
By (i), the fact that $P_1$'s move was possible in a given round implies that $P_2$'s move is also possible.
The desired tie follows from applying (ii) to the last round.
\end{proof}

\begin{lemma}
\label{lem.main}
Suppose $m>14j$.
Then in the $(j,m)$-balanced redistricting game, $P_1$ can force a tie.
\end{lemma}

Before describing a strategy for $P_1$, we need an important definition.
Immediately before $P_i$'s turn in round $r$, let $l$ be the number of majority-white bins, and select any $\min\{j,l\}$ of these bins with the most white balls along with any $\max\{j-l,0\}$ additional bins with the fewest total balls.
Let $S_{r,i}$ denote the set comprised of these $j$ bins.
Table~\ref{table.strategy} uses this definition to prescribe a strategy for $P_1$, and as we will show, this strategy forces a tie.
First, we show that this strategy prescribes legal moves for $P_1$:

\begin{table}[t]
\begin{framed}
If in round $r$, there exists a bin in $S_{r,1}$ that contains fewer than $m+1$ white balls, and furthermore, there is a white ball available to play, then:
\begin{itemize}
\item
If $r=1$, then $P_1$ plays the default move.
\item
If in round $r-1$, $P_2$ played white, then $P_1$ plays the default move.
\item
If in round $r-1$, $P_2$ played black in some bin $b\not\in S_{r-1,2}$, then:
\begin{itemize}
\item
If $b$ contains exactly one ball and there exists an empty bin $b'\in S_{r,1}$, then $P_1$ plays white in $b'$.
\item
Otherwise, $P_1$ plays the default move.
\end{itemize}
\item
If in round $r-1$, $P_2$ played black in some bin $b\in S_{r-1,2}$, then:
\begin{itemize}
\item
If $b$ contains exactly one ball and there exists an empty bin $b'\in S_{r,1}$, then $P_1$ plays white in $b'$.
\item
Else if $b$ contains fewer than $m+1$ white balls, then $P_1$ plays white in $b$.
\item
Otherwise, $P_1$ plays the default move.
\end{itemize}
\end{itemize}
Otherwise, $P_1$ plays arbitrarily.
\end{framed}
\caption{\label{table.strategy}Strategy for $P_1$ in $(j,m)$-balanced redistricting game. We say $P_1$ plays ``the default move'' in round $r$ if he plays white in any of the most full bins in $S_{r,1}$ with fewer than $m+1$ white balls.}
\end{table}

\begin{lemma}
For every round, Table~\ref{table.strategy} determines a legal move for $P_1$.
\end{lemma}

\begin{proof}
We only need to consider rounds $r$ for which there exists a bin in $S_{r,1}$ that contains fewer than $m+1$ white balls, and furthermore, a white ball is available for $P_1$ to play.
Let $A_{r,i}$ and $B_{r,i}$ denote the number of empty bins in $S_{r,i}$ and $S_{r,i}^c$, respectively.
We will use induction to prove that the following hold simultaneously:
\begin{itemize}
\item[(i)]
$A_{r,2}\leq B_{r,2}$.
\item[(ii)]
The bins in $S_{r,2}$ that are not majority-white are empty.
\item[(iii)]
In round $r$, the above strategy determines a legal move for $P_1$.
\end{itemize}
The claim clearly holds for $r=1$.
Now suppose it holds for a given $r\geq1$.

\textbf{Case I:} In round $r$, $P_2$ plays white.
If $P_2$ plays white in an empty bin $b$, then this bin becomes majority-white, and so by (ii) of the induction hypothesis (which we denote IH(ii) in the sequel), $b\in S_{r+1,1}$ unless all of the bins in $S_{r,2}$ are already majority-white.
After $P_1$ plays the default move, we either have $A_{r+1,2}=0\leq B_{r+1,2}$ or 
\[
A_{r+1,2}
\leq A_{r+1,1}
=A_{r,2}-1
\leq B_{r,2}-1
=B_{r+1,2}-1
\leq B_{r+1,2},
\]
where the second inequality applies IH(i).
If $P_2$ plays white in a nonempty bin, then IH(ii) implies $S_{r+1,1}=S_{r,2}$.
After $P_1$ plays the default move, we then have $A_{r+1,2}\leq A_{r,2}\leq B_{r,2}=B_{r+1,2}$, where the second inequality applies IH(i).
This proves (i).
Next, (ii) continues to hold for round $r+1$ since the balls played by $P_2$ and $P_1$ are both white. 
For (iii), note that IH(ii) and the above discussion together gives that the bins in $S_{r+1,1}$ are all either majority-white or empty.
Since by assumption, there exists a bin in $S_{r+1,1}$ that contains fewer than $n+1$ white balls, the default move is legal.

\textbf{Case II:} In round $r$, $P_2$ plays black in some bin $b\not\in S_{r,2}$.
Then $P_1$ responds in such a way that $S_{r+1,2}=S_{r+1,1}=S_{r,2}$.
If $P_2$ plays black in an empty bin so that $B_{r+1,1}=B_{r,2}-1$, then either $P_1$ plays white in an empty bin to make $A_{r+1,2}=A_{r+1,1}-1$, or there is no empty bin available, meaning $A_{r+1,2}=0\leq B_{r+1,2}$.
The former case gives
\[
A_{r+1,2}
=A_{r+1,1}-1
=A_{r,2}-1
\leq B_{r,2}-1
=B_{r+1,2}
=B_{r+1,2},
\]
where the inequality follows from IH(i).
Overall, we have (i).
Next, there is only one new ball in $S_{r+1,2}=S_{r+1,1}=S_{r,2}$, and it is white, so (ii) continues to hold.
Finally, (iii) follows from IH(ii), as in the previous case.

\textbf{Case III:} In round $r$, $P_2$ plays black in some bin $b\in S_{r,2}$.
If $P_2$ plays black in an empty bin, then by IH(i), there must be an empty bin outside $S_{r,2}$, resulting in an empty bin $b'\in S_{r+1,1}$, where $P_1$ plays white.
As such, (i)--(iii) continue to hold in this subcase.
If $b$ was nonempty before $P_2$ played black, then by IH(ii), $b$ is either majority-white or tied before $P_1$'s move.
If tied, then $b$ contains fewer than $m+1$ white balls, and so $P_1$ plays white in $b$, thereby regaining the majority.
If majority-white, then $P_1$ plays white in $b$ or some bin in $S_{r,1}$.
In particular, the default move is legal.
Regardless, (i)--(iii) continue to hold.
\end{proof}

Observe that any white ball beyond $m+1$ is unnecessary to obtain a majority in a bin.
As such, we say a bin containing $x$ white balls has $\max\{x,m+1\}$ \textbf{non-wasted white balls}.
Let $f(r)$ denote the number of non-wasted white balls in $S_{r,2}$ immediately before $P_2$'s turn in round $r$.
Observe that $f$ is monotonically increasing in $r$, and that $f(j(2m+1))=j(m+1)$ implies $P_1$ wins at least half of the bins.
In fact, under the above strategy, $f$ is \textit{strictly} increasing over rounds in which a white ball is available for $P_1$ to play:

\begin{lemma}
\label{lem.strictly increasing}
If there exists a bin in $S_{r+1,1}$ that contains fewer than $m+1$ white balls, and furthermore, there is a white ball for $P_1$ to play in round $r+1$, then Table~\ref{table.strategy} ensures
\[
f(r+1)\geq f(r)+1.
\]
\end{lemma}

\begin{proof}
Let $f_{r,i}$ denote the number of non-wasted white balls in $S_{r,i}$ immediately before $P_i$'s turn in round $r$, so that $f(r)=f_{r,2}$.
If $P_2$ plays white in round $r$, then $f_{r+1,1}\geq f_{r,2}$, and then $P_1$ plays in such a way that
\[
f(r+1)
=f_{r+1,2}
=f_{r+1,1}+1
\geq f_{r,2}+1
=f(r)+1.
\]
If in round $r$, $P_2$ plays black in some bin $b\not\in S_{r,2}$, then $S_{r+1,1}=S_{r,2}$, and then $P_1$ plays in such a way that
\[
f(r+1)
=f_{r+1,2}
=f_{r+1,1}+1
=f_{r,2}+1
=f(r)+1.
\]
If in round $r$, $P_2$ plays black in some bin $b\in S_{r,2}$, then $P_1$ similarly plays in such a way that
\[
f(r+1)
=f_{r+1,2}
=f_{r,2}+1
=f(r)+1.
\qedhere
\]
\end{proof}

As such, the only way $P_2$ can prevent $P_1$ from winning at least half of the bins is by playing enough white balls that none are available for $P_1$'s move in round $j(m+1)$.
Since there are $j(2m+1)$ white balls total, this means $P_1$ needs to play white for almost all of these first rounds; specifically, $P_2$ must play strictly fewer than $j-1$ black balls in the first $j(m+1)-1$ rounds.
As we will see, under this play, almost half of the white balls played by $P_2$ in the early game will necessarily contribute to $f(r)$, thereby failing to keep $P_1$ from winning half of the bins:

\begin{lemma}
\label{lem.early game}
Pick $r>2j$ and suppose $m>2r$.
If in the first $r$ rounds, $P_2$ plays at most $j$ black balls, then Table~\ref{table.strategy} ensures
\[
f(r+1)
\geq \bigg(1+\frac{j-1}{2j-1}\bigg)(r-j)+1.
\]
\end{lemma}

\begin{proof}
Consider the state of the game immediately before $P_2$'s turn in round $r+1$.
Let $l$ denote the total number of majority-white bins, let $e$ denote the total number of empty bins, let $w_1\geq\cdots\geq w_{l+e}$ denote numbers of white balls in these majority-white and empty bins, and let $W=\sum_{i=1}^{l}w_i$ denote the total number of white balls in majority-white bins.

At this point in the game, at most $j$ black balls have been played, and so at most $j$ white balls reside in bins that are not majority-white.
As such, $P_2$ contributed at least $r-2j$ white balls to majority-white bins, while $P_1$ has contributed $r+1$, meaning $W\geq 2r-2j+1$.
Also, $P_1$ played the default move in all but at most $j$ turns (i.e., when responding to $P_2$ playing black in certain ways), and so $w_1\geq r-j+1$.
Furthermore, $m>2r$ ensures that every white ball played up to this point is non-wasted, and so
\begin{align*}
f(r+1)
=w_1+\sum_{i=2}^k w_i
&\geq w_1+\frac{j-1}{l+e-1}\sum_{i=2}^{l+e} w_i\\
&=\bigg(1-\frac{j-1}{l+e-1}\bigg)w_1+\frac{j-1}{l+e-1}W\\
&\geq \bigg(1-\frac{j-1}{l+e-1}\bigg)(r-j+1)+\frac{j-1}{l+e-1}(2r-2j+1)\\
&=\bigg(1+\frac{j-1}{l+e-1}\bigg)(r-j)+1\\
&\geq \bigg(1+\frac{j-1}{2j-1}\bigg)(r-j)+1.
\hspace{1in}\qedhere
\end{align*}
\end{proof}

\begin{proof}[Proof of Lemma~\ref{lem.main}]
The result is trivial when $j=1$.
Now suppose $j\geq 2$.
By Lemma~\ref{lem.strictly increasing}, $P_1$ forces a tie if $P_2$ plays more than $j$ black balls in the first $j(m+1)-1$ rounds.
Now suppose $P_2$ plays at most $j$ black balls in these rounds.
Then Lemma~\ref{lem.early game} gives
\[
f(7j+1)
\geq \bigg(1+\frac{j-1}{2j-1}\bigg)\cdot 6j+1
\geq 8j+1,
\]
where the last step follows from the fact that $j\mapsto\frac{j-1}{2j-1}$ is an increasing function for $j\geq2$.
Furthermore, a white ball is necessarily available for $P_1$ to play in each of the first $jm$ rounds.
As such, the monotonicity of $f$, Lemma~\ref{lem.strictly increasing}, and the above inequality together give
\[
f(j(2m+1))
\geq f(jm)
\geq \Big(jm-(7j+1)\Big)+f(7j+1)
\geq j(m+1),
\]
meaning $P_1$ wins at least half of the bins, as desired.
\end{proof}

\section{Discussion}

In this paper, we introduce Utility Ghost, a game-based redistricting protocol in which two players take turns assigning atoms (that could be voting precincts, counties, or even individual houses) to districts.
The theoretical analysis in this paper proves that, in the non-geometrically constrained version of the game (where districts are not required to be contiguous, for example), if each player has half of the votes, then both players have a strategy to win half of the districts.
We also report numerical experiments in small-scale settings (where we can solve the game by naively implementing the minimax algorithm), showing that partisan symmetry is obtained under optimal play.

Note that the notion of optimal play assumes a particular utility function for each player, which in this case is simply the number of seats won.
In practice, players could have very different objectives; for instance, the majority party might aim to maximize the safety of its majority instead of maximizing the number of seats.
Such an objective implicitly requires a probabilistic model for the vote, and presumably, the atoms' random vote variables are highly correlated (cf.~\cite{Silver:18}).
One could also consider non-partisan objectives like favoring incumbents, or even a combination of different objectives.
In these cases, the protocol would not change, only the utility functions $u_1$ and $u_2$, but the game would not necessarily be zero-sum and interesting properties may arise. 

In order to analyze large-scale instances of Utility Ghost (say, with larger states, smaller atoms, and/or more districts), different techniques will be required.
Presumably, one could leverage state-of-the-art deep neural networks that learn how to play board games like chess and Go through self-play~\cite{SilverEtal:17}.
Such an approach would use a set of admissible maps to determine the set of valid moves at each turn and learn an evaluation function that is implemented by a neural network. 
The set of admissible maps should satisfy various (possibly state-specific) legal constraints such as one person--one vote, geographic compactness, and restrictions arising from the Voting Rights Act.
Markov Chain Monte Carlo sampling could be used to produce a small, but representative set of maps~\cite{BangiaEtal:17}.
The complexity of the game depends chiefly on the number of atoms, which equals the number of turns in the game.
It would be interesting to design relevant atoms that take political geography into consideration.

\section*{Acknowledgments}

The authors thank Boris Alexeev, Joseph Iverson and John Jasper for helpful feedback on an initial draft of this manuscript.
DGM was partially supported by AFOSR FA9550-18-1-0107, NSF DMS 1829955, and the Simons Institute of the Theory of Computing.
SV was partially supported by EOARD FA9550-18-1-7007 and by the Simons Algorithms and Geometry (A\&G) Think Tank.
The views expressed in this article are those of the authors and do not reflect the official policy or position of the authors' employers, the United States Air Force, Department of Defense, or the U.S.\ Government.

\end{document}